\documentclass[11pt]{amsart}
\usepackage[margin=1.30in]{geometry}
\usepackage{setspace}

\usepackage[utf8]{inputenc}
\usepackage{amsmath,amsfonts,amssymb,amsopn,amscd,amsthm,bbm}
\usepackage{comment}
\usepackage{dsfont}
\usepackage{graphicx}
\usepackage{color}
\usepackage[colorlinks]{hyperref}
\usepackage{epigraph,todonotes}
\usepackage{enumerate}
\allowdisplaybreaks[4]

\def\1{{\mathbf 1}}

\def\N{{\mathbb N}}

\def\R{{\mathbb R}}

\def\P{{\mathbb P}}
\def\E{{\mathbb E}}

\def\Nc{{\mathcal N}}

\def\Pc{{\mathcal P}}

\def\Sc{{\mathcal S}}

\numberwithin{equation}{section}

\newtheorem{proposition}{Proposition}[section]

\newtheorem{definition}{Definition}[section]

\newtheorem{lemma}{Lemma}[section]

\newtheorem{rmk}{Remark}[section]

\newcommand{\ignore}[1]{}
\newcommand{\vertiii}[1]{{\left\vert\kern-0.25ex\left\vert\kern-0.25ex\left\vert #1 
    \right\vert\kern-0.25ex\right\vert\kern-0.25ex\right\vert}}

\title{A smooth variational principle on Wasserstein space} 
\author{Erhan Bayraktar}\thanks{E. Bayraktar is partially supported by the National Science Foundation under grant DMS-2106556 and by
the Susan M. Smith chair.}
\address{Department of Mathematics, University of Michigan}
\email{erhan@umich.edu}
\author{Ibrahim Ekren}\thanks{I. Ekren is supported in part by NSF Grant DMS 2007826.}
\address{Department of Mathematics, Florida State University}
\email{iekren@fsu.edu}
\author{Xin Zhang} 
\address{Department of Mathematics, University of Vienna}
\email{xin.zhang@univie.ac.at}


\begin{document}
\maketitle

\begin{abstract}
    In this note, we provide a smooth variational principle on Wasserstein space by constructing a smooth gauge-type function using the sliced Wasserstein distance. This function is a crucial tool for optimization problems and in viscosity theory of PDEs on Wasserstein space. 
\end{abstract}

\section{Introduction} 
This note is devoted to proving a smooth variational principle on Wasserstein space. Due to the lack of local compactness, a continuous functions on  an infinite dimensional space may not attain its local maxima/minima, which becomes an issue when dealing with optimization problems. Smooth variational principle provides a way to perturb the function smoothly so that its perturbation can attain its local extremas. Recently,  smooth variational principles on Wasserstein space appeared in the study of viscosity solution of partial differential equations on Wasserstein space. A major effort in this direction was performed by \cite{cosso2021master}.

For a continuous function on a separable Hilbert space, Ekeland's variational principle provides a smooth variation so that the perturbation attains local extremas; see e.g. \cite{MR1013931}. However, in the Wasserstein space the variation part is given by the Wasserstein metric which is not smooth anymore. One of the observations in \cite{cosso2021master} was to use the Borwein-Preiss variational principle \cite[Theorem 2.5.2]{MR2144010}, to have smooth variations on the Wasserstein space, which states that it is sufficient to construct a topologically equivalent complete metric  which is differentiable in the sense of \cite{cdll2019}. In this note, we achieve this using the sliced Wasserstein distance, which defines a metric between high dimensional probability distributions using their one dimensional projections; see e.g. \cite{10.1214/21-ECP383} and page 214 of \cite{santambrogio2015optimal}. The advantage of our choice is that the optimal transport map in one dimension can be explicitly written down, and is regular after a Gaussian convolution. As such, our choice of the sliced Wasserstein distance allows a simple construction of smooth gauge type function compared to the alternative in \cite{cosso2021master}; see in particular Lemma 4.4 therein.

In the next subsection, we recall the definition of Wassertein distance, and the $L$-derivative. Then in Section 2, we analyze the differential properties of Gaussian regularized sliced Wasserstein distance, and finally prove the smooth variational principle in Proposition~\ref{prop:main}.

\subsection{Wasserstein distance and $L$ derivative}

We denote by $\Pc_2(\R^k)$ the set of Borel probability measures $\mu$ such that $\int |x|^2 \, \mu (dx)<\infty$. We endow the space $\Pc_2(\R^k)$  with the 2-Wasserstein distance $W_2$, i.e., for any $\mu,\nu \in \Pc_2(\R^k)$
\begin{align}\label{eq:wasdist}
W_2(\mu,\nu)^2:= \inf_{\pi \in \Pi(\mu,\nu)} \int \frac{1}{2}|x-y|^2 \, \pi(dx,dy),
\end{align}
where $\Pi(\mu,\nu)$ denotes the collection of probability measures on $\R^k \times \R^k$ with first and second marginals $\mu$ and $\nu$ respectively. 

Let us now present the $L$-derivative introduced in \cite{cdll2019}; see \cite[Chapter 5]{MR3752669} for a survey. Let $u:\Pc_2(\R^k) \to \R$, and $(\Omega, \mathcal{P}, \mathbb{P})$ be an atomless probability space. The lifting $U$ of $u$ on the Hilbert space $L^2(\Omega, \mathcal{P},\mathbb{P}; \R)$ is defined via 
\[
U(X):=u(\P_{X}), \quad \forall \, X \in L^2(\Omega,\mathcal{P},\mathbb{P}; \R), 
\]
where $\P_{X}$ stands for the distribution of $X$. Recall that $U$ is said to be Fr\'{e}chet differentiable at some random variable $X \in L^2(\Omega,\mathcal{P},\mathbb{P}; \R)$ if there exists a random variable $Z \in L^2(\Omega,\mathcal{P},\mathbb{P}; \R)$ such that 
\[
\lim\limits_{t \to 0} \frac{U(X+tY)-U(X)}{t}=\E[ZY], \quad \forall \, Y \in L^2(\Omega,\mathcal{P},\mathbb{P}; \R), 
\]
and we denote this derivative $Z$ by $D U(X)$.

\begin{definition}
A function $u: \Pc_2(\R^k) \to \R$ is said to be $L$-differentiable at $\mu$ if there exists some $X \in L^2(\Omega, \mathcal{P},\mathbb{P})$ such that $\P_{X}=\mu$ and $U$ is Fr\'{e}chet differentiable at $X$. And $u$ is said to be $L$-differentiable if there exists a jointly continuous function $D_{\mu} u: \Pc_2(\R^k) \times \R^k \to \R$ such that the lifting $U$ is Fr\'{e}chet differentiable at any $X \in L^2(\Omega, \mathcal{P},\mathbb{P}; \R)$ and $DU(X)=D_{\mu}u(\mu,X)$.
\end{definition}

It was proven in \cite[Theorem 2.2]{MR4174419} that if {the optimizers of \eqref{eq:wasdist} is reduced to the set $\{(Id,T)\sharp \mu\}$ for some measurable function $T$, then $\mu \mapsto W_2(\mu,\nu)^2$ is $L$-differentiable at $\mu$. Additionally, the Fr\'{e}chet derivative of its lift at $X$ with $X \sim \mu$ is $X-T(X)$.} By the Brenier's theorem the condition on the uniqueness of the optimizer is satisfied when $\mu$ is absolutely continuous.

\section{Gaussian regularized sliced Wasserstein distance}

 For any $\mu \in \Pc_2(\R^k)$, and $\theta \in \mathcal{S}^{k-1}$, define the mapping $P_\theta:\R^k\to \R$ by the expression $P_\theta (x)=x^\top \theta$ and the pushforward measure 
 $\mu_{\theta}:=P_\theta\sharp\mu\in  \Pc_2(\R)$. For any $\mu, \nu \in \Pc_2(\R^k)$, the sliced Wasserstein distance is defined via 
\begin{align}
SW_2(\mu,\nu)^2=\int W_2(\mu_{\theta}, \nu_{\theta})^2 \, d \theta,
\end{align}
where the integration is with respect to the standard spherical measure on $\mathcal{S}^{k-1}$; see e.g. \cite{10.1214/21-ECP383} and page 214 of \cite{santambrogio2015optimal}.
Moreover, we consider the Gaussian regularized version 
\begin{align}
SW_2^{\sigma}(\mu,\nu):=  SW_2(\mu^{\sigma}, \nu^{\sigma}),
\end{align}
where $\mu^{\sigma}:=\mu \ast \mathcal{N}_{\sigma}$ and $\mathcal{N}_{\sigma} \in \Pc_2(\R^k)$ is the Normal distribution with variance $\sigma^2 I_{k}$ for some $\sigma\in(0,\infty)$. By abuse of notation, $\mathcal{N}_{\sigma}$ also denotes the one dimensional normal distribution (and its density) with mean $0$ and variance $\sigma^2$, and then we have that $\left(\mu^{\sigma}\right)_{\theta}=\left(\mu \ast \mathcal{N}_{\sigma}\right)_{\theta}=\mu_{\theta} \ast \mathcal{N}_{\sigma\theta}=\mu_{\theta} \ast \mathcal{N}_{\sigma}$ where the last Gaussian is one-dimensional and the previous one is $k$-dimensional. 

\begin{lemma}\label{lem:1}
For any $\sigma \geq 0$, $(\Pc_2(\R^k),SW^{\sigma}_2)$ is a complete metric space, and it is equal to $(\Pc_2(\R^k), W_2)$ as a topological space. 
\end{lemma}
\begin{proof}

Let us first prove the first claim. Take any Cauchy sequence $(\mu^n)_{n \geq 1}$ in $(\Pc_2(\R^k), SW^{\sigma}_2)$, and we can assume without loss of generality that 
\begin{align*}
\sum_{n \geq 1} SW^{\sigma}_2(\mu^n, \mu^{n+1})^2=\int \sum_{n \geq 1} W_2(\mu^n_{\theta} \ast \Nc_{\sigma}, \mu^{n+1}_{\theta}\ast \Nc_{\sigma})^2 \, d\theta < +\infty.
\end{align*}
Define $\mathcal{S} \subset \mathcal{S}^{k-1}$ to be the the set of $\theta$ such that $\sum_{n \geq 1} W_2(\mu^n_{\theta}\ast \Nc_{\sigma}, \mu^{n+1}_{\theta} \ast \Nc_{\sigma})^2$ is  bounded. Then it is clear that $\mathcal{S} \subset \mathcal{S}^{k-1}$ is of full spherical measure. Choose a finite subset $\{\theta(1), \dotso, \theta(F) \} \subset \mathcal{S}$ with the property that 
\begin{align*}
| x |^2 \leq 2 \max_{i=1,\dotso, F} |x^\top \theta(i)|^2, \quad \forall \,   x \in \R^k.
\end{align*}
Then it can be easily seen that 
\begin{align*}
    \lim_{R \to \infty} \sup_{n \in \N} \int_{\{x \in \R^k: | x | \geq R\}} | x |^2 \, \mu^n \ast \Nc_{\sigma}(dx)=0,
\end{align*}
 and hence $\{\mu^n \ast \Nc_{\sigma} \}_{n \geq 1}$is tight with respect to $W_2$ topology.  As in \cite[Lemma 4.2]{cosso2021master}, it can be shown that $(\mu^n)_{n \geq 1}$ is also tight, and has a limit $\nu \in \Pc_2(\R^k)$ with respect to the $W_2$ metric. Due to the inequality 
\begin{align*}
    W_2(\mu^n_{\theta} \ast \Nc_{\sigma}, \nu_{\theta} \ast \Nc_{\sigma}) \leq W_2(\mu^n_{\theta}, \nu_{\theta}) \leq W_2(\mu^n, \nu),
\end{align*}
we conclude that $\mu^n$ converges to $\nu$ in  $SW^{\sigma}_2$ distance.  

The above inequality implies that the topology generated $W_2$ is stronger than that generated by $SW_2^{\sigma}$. By the argument in the first paragraph, for any sequence $(\mu^n)_{n \geq 1}$ such that $SW_2^{\sigma}(\mu^n, \nu) \to 0$ with some limit $\nu \in \Pc_2(\R^k)$, there is a tight subsequence that converges to $\nu$ in the $W_2$ distance. Therefore, $SW^{\sigma}_2$ induces the same topology as $SW_2$ and $W_2$.
\end{proof}

 The advantage of $SW^{\sigma}_2$ is that we can easily compute its derivatives. Denote the cumulative distribution function of $\mu$ by $F_{\mu}$. Then it is well known that in the one dimensional case, the optimal transport map from $\mu^{\sigma}_{\theta}$ to $\nu^{\sigma}_{\theta}$ is given by 
 \[
 T^{\sigma}_{\theta}(x):=F^{-1}_{\nu^{\sigma}_{\theta}}(F_{\mu^{\sigma}_{\theta}}(x)),
 \]
 which satisfies $W_2(\mu^{\sigma}_{\theta}, \nu^{\sigma}_{\theta})^2=\int \frac{1}{2} |x-T^{\sigma}_{\theta}(x)|^2 \, \mu^{\sigma}_{\theta}(dx)$.

\begin{lemma}\label{lem:2}
Let $ \nu \in \Pc_2(\R^k)$, $\sigma \geq 0$ and $\theta\in \Sc^{k-1}$ be fixed so that $F_{\mu^\sigma_{\theta}}$ and $ F_{\nu^\sigma_{\theta}}$ are continuous and strictly increasing functions\footnote{This assumption is only needed for $\sigma=0$.}. Then, the mapping
$$\mu\in \Pc_2(\R^k)\mapsto  SW^{\sigma}_2(\mu,\nu)^2$$
is $L$-differentiable, and
\begin{align}\label{eq:Dmusw}
&D_{\mu} SW^{\sigma}_2(\mu,\nu)^2(x) =  \int_{\mathcal{S}^{k-1}}\theta \left(\theta^\top x  -\E\left[ T^{\sigma}_{\theta}(\theta^\top  (x+N_\sigma)) \right]\right)\, d \theta,
\end{align}
where $N_{\sigma} \sim \mathcal{N}_{\sigma}$. Moreover, we have the estimate 
\begin{align}\label{eq:firstbound}
    \int_{\R^k} | D_{\mu} SW^{\sigma}_2(\mu,\nu)^2(x) |^2 \, \mu(dx) \leq C\left( \int_{\R^k} |x|^2 \, \mu(dx)+ \int_{\R^k} |y|^2 \, \nu^{\sigma}(dy) \right).
\end{align}
\end{lemma}
\begin{proof}
The proof relies on the proof of Theorem 2.2 of \cite{MR4174419}. We first prove that if $F_{\mu^\sigma_{\theta}}$ and $ F_{\nu^\sigma_{\theta}}$ are continuous and strictly increasing functions for some $\theta \in \Sc^{k-1}$, then the function 
$$\mu\in \Pc_2(\R^k)\mapsto W_2(\mu^\sigma_{\theta}, \nu^\sigma_{\theta})^2$$
is $L$-differentiable at $\mu$, and its $L$-derivatives is given by
\begin{align}\label{eq:dertheta}
D_{\mu}\left(W_2(\mu^\sigma_{\theta}, \nu^\sigma_{\theta})^2\right)(x)=\theta \left(\theta^\top x  -\E\left[ T^{\sigma}_{\theta}(\theta^\top  (x+N_\sigma)) \right]\right).
\end{align}

Fix $X\in L^2(\Omega, \mathcal{P},\mathbb{P}; \R)$ with distribution $\mu$ and $\xi\in L^2(\Omega, \mathcal{P},\mathbb{P}; \R)$ with norm $1$. Denote $N_\sigma\in L^2(\Omega, \mathcal{P},\mathbb{P}; \R)$ which is independent of $X$ with distribution $\mathcal{N}_{\sigma}$.
Denote $\mu^{n,\sigma}$ the distribution of $X+\frac{\xi}{n}+N_\sigma$ and note that $\mu^\sigma$ is the distribution of $X+N_\sigma$. 
By the minimality of the 2-Wasserstein distance, we have that 
\begin{align*}
    &W_2(\mu^{n,\sigma}_\theta,\nu_\theta^\sigma)^2\leq \frac{1}{2}\E\left[\left|X^\top\theta+\frac{\xi^\top\theta}{n}+N^\top_\sigma \theta-T^\sigma_{\theta}(\theta^\top(X+N_\sigma))\right|^2\right]\\
    &\leq  W_2(\mu^\sigma_\theta,\nu_\theta^\sigma)^2+\E\left[\frac{\xi^\top\theta}{n}\left((X+N_\sigma)^\top\theta-T^\sigma_{\theta}((X+N_\sigma)^\top\theta)\right)\right]+\frac{1}{2}\E\left[\frac{(\xi^\top\theta)^2}{n^2}\right].
\end{align*}
We now take $Y^n\in L^2(\Omega, \mathcal{P},\mathbb{P}; \R)$ with distribution $\nu_\theta^\sigma$ so that the coupling $(X^\top\theta+\frac{\xi^\top\theta}{n}+N^\top_\sigma\theta,Y^n)$ yields to an optimal coupling between $\mu^{n,\sigma}_\theta$ and $\nu_\theta^\sigma$. We have the following estimate
\begin{align*}
    W_2(\mu^\sigma_\theta,\nu_\theta^\sigma)^2&\leq \frac{1}{2}\E\left[\left|(X+N_\sigma)^\top\theta-Y^n\right|^2\right]\\
    &\leq  W_2(\mu^{n,\sigma}_\theta,\nu_\theta^\sigma)^2-\E\left[\frac{\xi^\top\theta}{n}\left((X+N_\sigma)^\top\theta-Y^n\right)\right]-\frac{1}{2}\E\left[\frac{(\xi^\top\theta)^2}{n^2}\right]\\
    &\leq  W_2(\mu^{n,\sigma}_\theta,\nu_\theta^\sigma)^2-\E\left[\frac{\xi^\top\theta}{n}\left((X+N_\sigma)^\top\theta-T^\sigma_{\theta}\left((X+N_\sigma)^\top\theta\right)\right)\right]\\
    & \ \ \ \ \ -\E\left[\frac{\xi^\top\theta}{n}\left(T^\sigma_{\theta}\left((X+N_\sigma)^\top\theta\right)-Y^n\right)\right]-\frac{1}{2}\E\left[\frac{(\xi^\top\theta)^2}{n^2}\right].
\end{align*}
Thus, we obtain the inequality 
\begin{align*}
    &n\left| W_2(\mu^{n,\sigma}_\theta,\nu_\theta^\sigma)^2-W_2(\mu^\sigma_\theta,\nu_\theta^\sigma)^2-\E\left[\frac{\xi^\top\theta}{n}\left((X+N_\sigma)^\top\theta-T^\sigma_{\theta}\left((X+N_\sigma)^\top\theta\right)\right)\right]\right|\\
    &\leq \E\left[\left|T^\sigma_{\theta}\left((X+N_\sigma)^\top\theta\right)-Y^n\right|^2\right]^{1/2}+\frac{1}{2n},
\end{align*}
where the last line goes to $0$ thanks to Lemma 2.5 of \cite{MR4174419}, and we obtain \eqref{eq:dertheta} thanks to the fact that $N_\sigma$ has $0$ mean. 

Integrating \eqref{eq:dertheta} over $\theta$, we obtain \eqref{eq:Dmusw}. In the end, let us show \eqref{eq:firstbound}, \begin{align*}
    \int_{\R^k} | D_{\mu} SW^{\sigma}_2(\mu,\nu)^2(x) |^2 \, \mu(dx) &\leq 2 \int \theta \theta^{\top} \, d \theta \int |x|^2 \, \mu(dx)+ 2\int \, d\theta \int |T^{\sigma}_{\theta}(x)|^2 \, \mu^{\sigma}_{\theta}(dx) \\
    &=2 \int \theta \theta^{\top} \, d \theta \int |x|^2 \, \mu(dx)+2 \int \, d\theta \int |y|^2 \, \nu^{\sigma}_{\theta}(dy) \\ &\leq C \left(\int_{\R^k} |x|^2 \, \mu(dx)+ \int_{\R^k} |y|^2 \, \nu^{\sigma}(dy) \right).
\end{align*}

\end{proof}

It can be easily seen that there exists some positive $\kappa \leq 1$ such that
\begin{align}\label{eq:kappa}
\int_{\mathcal{S}^{k-1}} \theta  \theta^\top \, d \theta = \kappa I_{k},  
\end{align}
and hence 
\begin{align*}
     \int_{\mathcal{S}^{k-1}} |\theta^\top x|^2 \, d \theta = \kappa |x|^2, \quad \forall x \in \R^k. 
\end{align*}
In the following Lemma, we will compute the derivative of $x \mapsto D_{\mu} SW^{\sigma}_2(\mu,\nu)(x)$ and their moments.
\begin{lemma}\label{lem:3}
For $\sigma>0$, we have the following results for derivatives. 
\begin{align}
     & D_{x\mu}^2  SW^{\sigma}_2(\mu,\nu)^2(x)\label{eq:Dmxusw} = \E \int_{\mathcal{S}^{K-1}} \theta  \theta^\top \left(1-(T_{\theta}^{\sigma})'(\theta^\top(x+N_{\sigma}) ) \right) d \theta, \\
 \int_{\R^k} & | D^2_{x\mu} SW^{\sigma}_2(\mu,\nu)^2(x) | \, \mu(dx) \leq C\left(1+\frac{1}{\sigma}\sqrt{ \int |y|^2 \, \nu^{\sigma}(dy)}\right). \label{eq:secondbound}
\end{align}
\end{lemma} 
\begin{proof}
Taking derivatives of $D_{\mu}SW^{\sigma}_2(\mu,\nu)^2(x)$ in $x$, we directly obtain \eqref{eq:Dmxusw}. Let us integrate $D^2_{x\mu}SW^{\sigma}_2(\mu,\nu)^2(x)$ over $\mu$. Using \eqref{eq:Dmxusw} and noting that $(T_{\theta}^{\sigma})'(x)$ is non negative, we have 
\begin{align}\label{eq:derivative}
    \int |D^2_{x\mu}SW^{\sigma}_2(\mu,\nu)^2(x)| \, \mu(dx) \leq C+ C\int \, d \theta \int  (T_{\theta}^{\sigma})'(x)  \, \mu^{\sigma}_{\theta}(dx). 
\end{align}
We only need the estimate of $ \int (T_{\theta}^{\sigma})'(x) \, \mu_{\theta}^{\sigma}(dx)$. According to the definition of convolution, we simply have that
\begin{align*}
\int (T_{\theta}^{\sigma})'(x) \, \mu_{\theta}^{\sigma}(dx) &= \int \mu_{\theta}(dx) \int (T_{\theta}^{\sigma})'(x+y) \Nc^{\sigma}(y) \, dy=- \int \mu_{\theta}(dx) \int T^{\sigma}_{\theta}(x+y) \frac{d}{dy} \Nc^{\sigma}(y) \, dy  \\
&=- \frac{1}{\sigma^2} \int \mu_{\theta}(dx)  \int y T^{\sigma}_{\theta}(x+y) \Nc^{\sigma}(y)\, dy. 
\end{align*}
Then the Cauchy-Schwartz inequality yields to  
\begin{align*}
& \left| \int \mu_{\theta}(dx)  \int y T^{\sigma}_{\theta}(x+y) \Nc^{\sigma}(y)\, dy \right| \\
& \leq \sqrt{\int \int |y|^2 \Nc^{\sigma}(y) \, dy \, \mu_{\theta}(dx)}\sqrt{\int \int |T^{\sigma}_{\theta}(x+y)|^2 \Nc^{\sigma}(y) \, dy \, \mu_{\theta}(dx)}  \\
&= \sigma \sqrt{\int |x|^2 \, \nu^{\sigma}_{\theta}(dx)},
\end{align*}
and therefore $\left|\int (T_{\theta}^{\sigma})'(x) \, \mu_{\theta}^{\sigma}(dx) \right|\leq \frac{1}{\sigma}\sqrt{ \int |y|^2 \, \nu^{\sigma}(dy)}$ for all $\theta \in \mathcal{S}^{k-1}$. Plugging this inequality into \eqref{eq:derivative}, we conclude that 
\begin{align*}
     \int |D^2_{x\mu}SW^{\sigma}_2(\mu,\nu)^2(x)| \, \mu(dx) \leq C\left(1+\frac{1}{\sigma}\sqrt{ \int |y|^2 \, \nu^{\sigma}(dy)}\right).
\end{align*}
\end{proof}

For each $\sigma>0$, let us define the function $\rho_{\sigma}:([0,T] \times \Pc_2(\R^k))^2 \to \R$ via 
\begin{align}
    \rho_{\sigma}\left((s,\mu),(t,\nu) \right)=|t-s|^2+SW^{\sigma}_2(\mu,\nu)^2.
\end{align}
Then it is a gauge type function on $(\Pc_2(\R^k),SW^{\sigma}_2)$; see \cite[Definition 2.5.1]{MR2144010}. The following smooth variational principle is the main result of this paper. 

\begin{proposition}\label{prop:main}
Fix $\delta>0$ and let $G: [0,T] \times \Pc_2(\R^k) \to \R$ be upper semicontinuous and bounded from above. Given $\lambda>0$, let $(t_0,\mu_0) \in [0,T] \times \Pc_2(\R^k)$ be such that 
\begin{align*}
    \sup_{(t,\mu) \in [0,T] \times \Pc_2(\R^k)}G(t,\mu) -\lambda \leq G(t_0,\mu_0).
\end{align*}
Then there exists $(\tilde t, \tilde \mu) \in [0,T] \times \Pc_2(\R^k)$ and a sequence $\{(t_n,\mu_n)\}_{n \geq 1} \subset [0,T] \times \Pc_2(\R^k)$ such that: 
\begin{enumerate}[(i)]
    \item $\rho_{1/\delta}((\tilde t, \tilde \mu),(t_n,\mu_n)) \leq \frac{\lambda}{2^n \delta^2}$, for every $n$;
    \item $G(t_0,\mu_0) \leq G(\tilde t, \tilde \mu)-\delta^2 \phi_{\delta}(\tilde t, \tilde \mu)$, with $\phi_{\delta}:[0,T] \times \Pc_2(\R^k) \to [0,+\infty)$ given by 
    \begin{align}\label{eq:smooth}
        \phi_{\delta}(t,\mu)=\sum_{n=0}^{+\infty} \frac{1}{2^n} \rho_{1/\delta}((t,\mu),(t_n,\mu_n)), \quad \forall \, (t,\mu) \in [0,T] \times \Pc_2(\R^k);
    \end{align}
    \item $G(t,\mu) -\delta^2 \phi_{\delta}(t,\mu) < G(\tilde t, \tilde \mu)-\delta^2 \phi_{\delta}(\tilde t, \tilde \mu)$, for every $(t,\mu) \in ([0,T] \times \Pc_2(\R^k))\setminus\{(\tilde t, \tilde \mu)\}$.
\end{enumerate}
Furthermore, the function $\phi_{\delta}$ satisfies the following properties: 
\begin{enumerate}
    \item[(1)] $\phi_{\delta}$ is differentiable in time and measure;
    \item[(2)] its time derivative is bounded by $4T$;
    \item[(3)] its measure derivative is bounded by 
    \begin{align*}
        \int |D_{\mu} \phi_{\delta}(t,\mu)(x)|^2 \, \mu(dx) \leq C \left(\int |x|^2 \, \mu(dx)+\int |x|^2 \, \tilde \mu(dx)+\frac{1}{\delta^2}  \right);
    \end{align*}
    \item[(4)] the derivative $D^2_{x\mu} \phi_{\delta}(t,\mu)(x)$ satisfies 
    \begin{align*}
        \int |D^2_{x\mu} \phi_{\delta}(t,\mu)(x)| \, \mu(dx) \leq C \left(1+\delta\sqrt{\int |x|^2 \, \tilde \mu(dx)} \right).
    \end{align*}
\end{enumerate}
\end{proposition}
\begin{proof}
Part \emph{(i),(ii),(iii)} directly follows from \cite[Theorem 2.5.2]{MR2144010}. Part $\emph{(1)}$ is due to Lemma~\ref{lem:2} and part $\emph{(2)}$ is trivial. Let us prove part $\emph{(3),(4)}$. 

Recall the $\kappa$ defined in \eqref{eq:kappa}. For any $\mu,\nu \in \Pc_2(\R^k)$, it can be easily seen that
\begin{align*}
    \kappa \int |y|^2 \, \nu(dx)=  \int \, d\theta \int |y|^2 \, \nu_{\theta}(dx) \leq \int |x|^2 \, \mu(dy)+SW_2(\mu,\nu)^2.
\end{align*}
Denoting $\sigma:=1/\delta$ and replacing $\nu$ and $\mu$ with $\mu_n^{\sigma}$ and $\tilde \mu^{\sigma}$ respectively in the above inequality, we obtain  $\kappa \int |y|^2 \, \mu_n^{\sigma}\leq \int |x|^2 \tilde \mu^{\sigma}(dx)+\frac{\lambda}{2^n \delta^2}$ by part \emph{(i)}. Therefore, according to Lemma~\ref{lem:3}, we get that 
\begin{align*}
    &\int |D_{\mu} \rho_{\sigma}((t,\mu),(t_n,\mu_n))(x)|^2 \, \mu(dx) \leq C \left(\int |x|^2 \, \mu(dx)+\int |x|^2 \, \tilde \mu(dx)+\frac{1}{\delta^2}  \right), \\
    &\int |D^2_{x\mu} \rho_{\sigma}((t,\mu),(t_n,\mu_n))(x)| \, \mu(dx) \leq C \left(1+\delta\sqrt{\int |x|^2 \, \tilde \mu(dx)} \right).
\end{align*}
Then by the definition \eqref{eq:smooth}, summing the above two inequalities over $n$ we conclude part $\emph{(3),(4)}$.
\end{proof}

\begin{rmk} 
\cite[Lemma 4.4]{cosso2021master} constructed a gauge type function using dyadic partitions of the underlying space $\R^k$. Our construction   $\rho_{\sigma}$ is much simpler, and can serve as a substitute of \cite[Lemma 4.4]{cosso2021master}. Furthermore, $\rho_{\sigma}$ is twice differentiable with respect to $\mu$, and thus could be useful in the study of second-order partial differential equations on Wasserstein space. 
\end{rmk}

\bibliographystyle{siam}
\bibliography{ref.bib}

\end{document}